\documentclass[12pt]{amsart}

\usepackage{amssymb}

\usepackage{enumitem}

\usepackage{graphicx}

\makeatletter
\@namedef{subjclassname@2020}{%
  \textup{2020} Mathematics Subject Classification}
\makeatother

\usepackage[T1]{fontenc}

\newtheorem{theorem}{Theorem}[section]

\newtheorem{lemma}[theorem]{Lemma}
\newtheorem{proposition}[theorem]{Proposition}

\theoremstyle{definition}

\newtheorem{remark}[theorem]{Remark}

\numberwithin{equation}{section}

\newtheorem{conjecture}[theorem]{Conjecture}
\def\bR{\mathbb{R}}
\def\cP{\mathcal{P}}
\def\cI{\mathcal{I}}
\def\cL{\mathcal{L}}
\def\iO{\textit{O}}

\begin{document}


\baselineskip=17pt


\title[Asymmetric estimates and the sum-product problems]{Asymmetric estimates and the sum-product problems}

\author[B. Xue]{Boqing Xue}
\address{ShanghaiTech University\\
393 Middle Huaxia Road\\
201210 Shanghai, China}
\email{xuebq@shanghaitech.edu.cn}

\date{}

\begin{abstract}
We show two asymmetric estimates, one on the number of collinear triples and the other on that of solutions to $(a_1+a_2)(a_1^{\prime\prime\prime}+a_2^{\prime\prime\prime})=(a_1^\prime+a_2^\prime)(a_1^{\prime\prime}+a_2^{\prime\prime})$.
As applications, we improve results on difference-product/division estimates and on Balog-Wooley decomposition: For any finite subset $A$ of $\bR$,
\[
\max\{|A-A|,|AA|\} \gtrsim |A|^{1+105/347},\quad \max\{|A-A|,|A/A|\} \gtrsim |A|^{1+15/49}.
\]
Moreover, there are sets $B,C$ with $A=B\sqcup C$ such that
\[
\max\{E^+(B),\, E^\times (C)\} \lesssim |A|^{3-3/11}.
\]
\end{abstract}



\maketitle

\section{Introduction}

In additive combinatorics, the Szemer\'{e}di-Trotter theorem plays an important role. For a set $A$, we use both $|A|$ and $\#A$ to denote the cardinality of $A$. For a set $\cP$ of points and a set $\cL$ of lines in $\bR^2$, define their incidences by $\cI(\cP,\cL)=\{(p,l)\in \cP\times \cL:\, p\in l\}$.

\begin{theorem}[Szemer\'{e}di-Trotter] \label{thm_ST}
Let $\cP$ be a finite set of points and $\cL$ be a finite set of lines in $\bR^2$. Then
\begin{equation} \label{eq_ST_2/3}
\cI(\cP,\cL)\leq 4|P|^{2/3}|L|^{2/3}+4|P|+|L|.
\end{equation}
\end{theorem}

A consequence of the Szemer\'{e}di-Trotter theorem shows an upper bound for the number of collinear triples in the plane. For three finite sets $A_1,A_2,A_3\subseteq \bR$, let $T^o(A_1,A_2,A_3)$ be the number of collinear triples $(u_1,u_2,u_3)$, where $u_i\in A_i\times A_i$ $(i=1,2,3)$ and $u_i\neq u_j$ $(1\leq i<j\leq 3)$. Moreover, define
\begin{align*}
T(A,B,C)=\#\{(a_1,a_2,b_1,b_2,&c_1,c_2)\in A\times A\times B\times B\times C\times C:\, \\
& (b_1-a_1)(c_2-a_2)=(c_1-a_1)(b_2-a_2)\}.
\end{align*}
Denote $T^o(A)=T^o(A,A,A)$ and $T(A)=T(A,A,A)$ for simplicity. Then (see \cite[Corollary 8.9]{TV})
\begin{equation} \label{eq_sym_1}
T(A)\lesssim |A|^4
\end{equation}
for any finite set $A$. Here $a\lesssim b$ means that $a=\iO\left(b(\log b)^c\right)$ for some absolute constant $c>0$. And we write $a\approx b$ if both $a\lesssim b$ and $b\lesssim a$ hold.

When $|A_1|\leq |A_2|\leq |A_3|$, Shkredov \cite{Shk17} showed an asymmetric bound
\[
T(A_1,A_2,A_3) \lesssim |A_2|^2|A_3|^2.
\]
In this paper, we improve this bound as the following.
\begin{theorem} \label{thm_asy_1}
Suppose that $A_1,A_2,A_3$ are three finite subsets of $\bR$ with $|A_1|\leq |A_2|\leq |A_3|$. Then
\begin{equation} \label{eq_asy_1}
T^o(A_1,A_2,A_3) \lesssim |A_1||A_2|^{5/3}|A_3|^{4/3}.
\end{equation}
Moreover,
\[
T(A_1,A_2,A_3) \lesssim |A_1||A_2|^{5/3}|A_3|^{4/3}+|A_1|^2|A_3|^2.
\]
\end{theorem}

The upper bound in \eqref{eq_asy_1} is not symmetric in $|A_1|,|A_2|,|A_3|$. One might expect to get a bound like $|A_1|^{4/3}|A_2|^{4/3}|A_3|^{4/3}$. However, this symmetric estimate does not hold in general. Let $A$ be a set such that $0\notin A$ and $E^\times(A)\gg |A|^3$, where
\[
E^\times (A) =\{(a_1,a_2,a_3,a_4)\in A^4:\, a_1a_2=a_3a_4\}.
\]
Take $A_1=\{0\}$ and $A_2=A_3=A$. Then
\[
(0,0),\quad (a,b),\quad (c,d),
\]
where $a,b,c,d\in A$, $ad=bc$ and $a\neq c$, are collinear triples counted by $T^o(\{0\},A,A)$. Clearly, we have
\[
T^o(\{0\},A,A) \geq E^\times(A)-|A|^2 \gg |A|^3,
\]
while $1^{4/3}|A|^{4/3}|A|^{4/3} = |A|^{8/3}$. The asymmetric bound \eqref{eq_asy_1} is tight in this example.

\medskip

Next, let us study the number $R(Z;A_1,A_2)$ of solutions to
\[
\frac{a_1+a_2}{a_1^\prime+a_2^\prime}=\frac{a_1^{\prime\prime}+a_2^{\prime\prime}}{a_1^{\prime\prime\prime}+a_2^{\prime\prime\prime}}\in Z,
\]
where $a_i,a_i^\prime,a_i^{\prime\prime},a_i^{\prime\prime\prime}\in A_i$ $(i=1,2)$. In \cite{MRS14}, Murphy, Roche-Newton and Shkredov proved a symmetric estimate:
\begin{equation} \label{eq_sym_2}
R(\bR;A_1,A_2) \lesssim |A_1|^3|A_2|^3.
\end{equation}
This bound is tight when $A_1=\{0\}$ and $A_2$ is a finite set satisfying $E^\times(A_2)\gg |A_2|^3$. However, by adding a constrain on the value set $Z$, we can prove a better asymmetric estimate.

\begin{theorem} \label{thm_asy_2}
Let $A_1,A_2,Z$ be finite subsets of $\bR$ such that $|A_1|\leq |A_2|$ and $|Z|\lesssim |A_1|^2$. Then
\begin{equation} \label{eq_asy_2}
R(Z;A_1,A_2) \lesssim |A_1|^{10/3}|A_2|^{8/3}.
\end{equation}
\end{theorem}

\noindent We will apply these two asymmetric bounds in sum-product type problems.

\medskip

As a ring, the set $\bR$ of real numbers does not contain non-trivial finite subrings. So any finite subset of $\bR$ can not have good additive structure and good multiplicative structure simultaneously. This phenomenon is characterized by the folklore Erd\H{o}s-Szemer\'{e}di sum-product conjecture \cite{ErSz83}.

\begin{conjecture} [Erd\H{o}s-Szemer\'{e}di] \label{conj}
Let $\delta< 1$ be given. Then for any finite $A\subseteq \bR$, one has
\begin{equation} \label{eq_ET_conj}
\max\{|A+A|,\,|AA|\} \gtrsim |A|^{1+\delta}.
\end{equation}
\end{conjecture}

\noindent Here, for two sets $A$, $B$ and an operator $\circ\in \{+,-,\cdot,/\}$,
\[
A\circ B =\{a\circ b:\, a\in A,\,b\in B\}.
\]
And we abbreviate $A\cdot A$ by $AA$ for simplicity. When $A$ is an arithmetic progression, then $|A+A|\ll |A|$ and $|AA|\gtrsim |A|^2$. On the other side, when $A$ is a geometric progression, then $|AA|\ll |A|$ and $|A+A|\gtrsim |A|^2$.

For two sets $A$ and $B$, define
\[
r_{A\circ B}(x) =\#\{(a,b)\in A\times B:\, a\circ b=x\}.
\]
Other quantities that reflects the structure of sets $A$ and $B$ include energies
\[
E_k^+(A,B) = \sum\limits_{x}r_{A-B}^k(x),\quad E_k^\times(A,B) = \sum\limits_{x}r_{A/B}^k(x)
\]
for $k>1$. (Indeed, one may also regard $E_0^+(A,B)=|A+B|$.) We write $E_k^+(A)=E_k^+(A,A)$ and $E_k^\times(A)=E_k^\times(A,A)$ for simplicity. In particular, one has $E^\times (A)=E_2^\times(A)$, and we denote $E^+(A)=E^+_2(A)$. Moreover, for $k>1$, the quantities
\[
d^+_k(A) =\sup\limits_{B\neq \emptyset} \frac{E_k^+(A,B)}{|A||B|^{k-1}},\quad d^\times_k(A) =\sup\limits_{B\neq \emptyset} \frac{E_k^\times(A,B)}{|A||B|^{k-1}}
\]
are also quite helpful in describing the structure of a set $A$. Especially, the quantity $d_3^+$ is related to so-called Szemer\'{e}di-Trotter set. When $A$ has strong additive structure, quantities $E_k^+(A)$ are large. When $A$ has weak additive structure, one might expect them to be small. However, this is not the case. Consider a set $A=B\sqcup C$. Suppose that $B$ has weak additive structure, and $C$ has strong additive structure but very small cardinality. Then $E_k^+(B)$ can be small. But for energies with higher moments, the contribution of $C$ in $E_k^+(C)$ may be significantly enhanced. And $E_k^+(A)$ might be large as well. So a suitable way to consider energies in sum-product type problems, known as Balog-Wooley decomposition \cite{BalWoo17}, is to decompose $A$ into two sets, one with little additive structure and the other with little multiplicative structure.

\begin{conjecture} [Balog-Wooley] \label{conj_2}
Let $\eta\leq 2/3$ be given. Then for any finite set $A\subseteq \bR$, there are sets $B,C$ such that $A=B\sqcup C$ and
\begin{equation} \label{eq_BW_conj}
\max\{E^+(B),\, E^\times (C)\} \lesssim |A|^{3-\eta}.
\end{equation}
\end{conjecture}

\noindent The example $\{(2m-1)2^j:\, 1\leq m\leq S,\, 1\leq j\leq P\}$ (see \cite{BalWoo17}) shows that the exponent $\eta=2/3$ in \eqref{eq_BW_conj} is tight.


\medskip

Now we introduce current results on Conjecture \ref{conj}. Erd\H{o}s and Szemer\'{e}di \cite{ErSz83} showed that \eqref{eq_ET_conj} holds for some $\delta>0$. Nathanson \cite{Nat97} obtained the first quantitative estimate $\delta=1/31$, which was improved by Ford \cite{For98} to $\delta=1/15$. The application of Szemer\'{e}di-Trotter theorem to study the sum-product type problems was started by Elekes \cite{Ele97}, who obtained \eqref{eq_ET_conj} with $\delta=1/4$. Later, Solymosi \cite{Sol05} considered more complicated incidences and attained $\delta=3/11$. By that time, the best bounds were same if the set $A+A$ is replaced by $A-A$, or/and the set $AA$ is replaced by $A/A$.

In \cite{Sol09}, Solymosi obtained \eqref{eq_ET_conj} with $\delta=1/3$. His approach uses geometry very cleverly: every ratio in $A/A$ corresponds to a line in $\bR^2$ passing the origin, and the addition of points (viewed as vectors) from a pair of lines contribute a point in $(A+A)^2$. In particular, if the pair of lines are carefully chosen, we can obtain points in $(A+A)^2$ without repetition. The estimate $\delta= 1/3$ in \eqref{conj} can be viewed as a standard exponent in the literature. See \cite{LiShen10} and \cite{Xue15} for other related results with $\delta=1/3$. We also remark that this approach does not work for the difference-set $A-A$, since the vector sums only allows the form $(A+A)\times (B+B)$ for two different sets $A$ and $B$.

Until \cite{KonShk15}, Konyagin and Shkredov modified Solymosi's approach and counted some repetitions. They broke the exponent $1/3$ and arose a new round of study. In a later paper \cite{KonShk16}, they showed $\delta=1/3+5/9813$. To prove the sum-product estimates, the symmetric estimate \eqref{eq_sym_1} was applied to obtain
\begin{equation} \label{eq_to_improve_1}
|AA|\gtrsim E^+(A)^4|A|^{-10}.
\end{equation}
A further improvement $\delta=1/3+1/1509$ were obtained by Rudnev, Shkredov and Stevens \cite{RSS16}. The symmetric estimate \eqref{eq_sym_2} is applied to obtain the following improvement of \eqref{eq_to_improve_1}:
\begin{equation} \label{eq_to_improve_2}
|AA|\gtrsim E^+(A)^3|A|^{-7}.
\end{equation}
Indeed, if we use the asymmetric estimate in Theorem \ref{thm_asy_1} in the proof of \eqref{eq_to_improve_1}, then we can obtain \eqref{eq_to_improve_2} directly. Later, Shakan \cite{Sha18} observed that the arguments in \cite{RSS16} actually deals with $E_4^+(A)$. The inequality he used to get $\delta=5/5277$ has the strength compared to
\begin{equation} \label{eq_to_improve_3}
|AA|\gtrsim E_4^+(A)|A|^{-3},
\end{equation}
which can be obtained from Lemma 5.6 of \cite{Sha18}.
At the same time when this paper was finished, Rudnev and Stevens \cite{RudSte20} showed an exciting update on the sum-product bound: $\delta=2/1167$.

Recall the term $|\cP|^{2/3}|\cL|^{2/3}$ on the right-hand side of \eqref{eq_ST_2/3}. The full application of Szemer\'{e}di-Trotter theorem should give the third moment energy. Since an inequality like \eqref{eq_to_improve_1}, \eqref{eq_to_improve_2} or \eqref{eq_to_improve_3} is a key step in the study of Conjecture \ref{conj}, we apply the asymmetric estimate in Theorem \ref{thm_asy_2} to show the following theorem in terms of the third moment energy.

\begin{theorem} \label{thm_E_3^+}
Let $A$ be a finite subset of $\bR$. Then
\begin{equation} \label{eq_to_improve_4}
|AA| \gtrsim E_3^+(A)^{4/3} |A|^{-10/3}.
\end{equation}
\end{theorem}

As a comparison, one may use $E_3^+(A)\geq E^+(A)^2|A|^{-2}$ and the trivial bound $E^+(A)\leq |A|^3$, to see that \eqref{eq_to_improve_4} is better than \eqref{eq_to_improve_2}. Moreover, since $E_3^+(A)^{4/3}\leq E_4^+(A)|A-A|^{1/3}$, the inequality \eqref{eq_to_improve_3} implies that $|AA|\gtrsim E_3^+(A)^{4/3}|A|^{-3}|A-A|^{-1/3}$, which is not better then \eqref{eq_to_improve_4}. In the study of Conjecture \ref{conj}, the effect of \eqref{eq_to_improve_4} is same as that of \eqref{eq_to_improve_3}, which gives $\delta=5/5277$. However, the inequalities \eqref{eq_to_improve_2}-\eqref{eq_to_improve_4} actually appear with $E^\times(A)$ instead of $|AA|$. And we can get some improvements on Conjecture \ref{conj_2}, and also on the difference-product/division problems.

For the difference-product/division problems
\[
\max\{|A-A|,|AA|\} \gtrsim |A|^{1+\kappa},\quad \max\{|A-A|,|A/A|\} \gtrsim |A|^{1+\rho},
\]
the first breakthrough after $\kappa=\rho=3/11$ by Solymosi \cite{Sol05} came from Konyagin and Rudnev \cite{KonRud13} with $\kappa=11/39$ and $\rho=9/31$. Shakan \cite{Sha18} improved the results to $\kappa=7/24$ and $\rho=3/10$. Shkredov's spectral method \cite{Shk13} plays an important role in these improvements (also see \cite{MRSS19,OSS19}). In this paper, we prove the following theorem.

\begin{theorem} \label{thm_difference}
Let $A$ be a finite subset of $\bR$. Then
\[
\max\{|A-A|,|AA|\} \gtrsim |A|^{1+105/347},\quad \max\{|A-A|,|A/A|\} \gtrsim |A|^{1+15/49}.
\]
\end{theorem}

\medskip

Now we introduce current results on Conjecture \ref{conj_2}. Balog and Wooley \cite{BalWoo17} first showed that \eqref{eq_BW_conj} holds with $\eta=2/33$. Later, Konyagin and Shkredov \cite{KonShk16} obtained $\eta=1/5$. They introduced a more flexible quantity $d_\ast(A)$, and avoided of using the Balog-Szemer\'{e}di-Gowers theorem. Their approach has been adopted by later studies. In \cite{RSS16}, Rudnev, Shkredov and Stevens proved $\eta=1/4$ (also see \cite{Shk18}). And Shakan \cite{Sha18} used higher energy decompositions to obtain $\eta=7/26$. In this paper, we obtain a further slight improvement.

\begin{theorem} \label{thm_BalWoo_improvement}
For any finite set $A\subseteq \bR$, there are some sets $B,C$ such that $A=B\sqcup C$ and
\[
\max\{E^+(B),\, E^\times (C)\} \lesssim |A|^{3-3/11}.
\]
\end{theorem}

We also mention that decompositions with higher moments were also studied by Shkredov \cite{Shk18}. He showed that
\[
\max\{E_3^+(B),\,\max E_3^\times(C)\} \lesssim |A|^{4-2/5},\quad \max\{d_3^+(B),\,d_3^+(C)\}\lesssim |A|^{1-2/5}.
\]
In \cite{Sha18}, the exponents $-2/5$ on $|A|$ were improved to $-1/2$.

The decomposition $A=B\sqcup C$ in Conjecture \ref{conj_2} does not ensure that both $B$ and $C$ constitute a positive proportion of $A$. In \cite{RSS16}, it is proved that there are disjoint subsets $X,Y$ of $A$ with $|X|,|Y|\geq |A|/3$ such that
\[
E^+(X)E^\times(Y) \lesssim |A|^{11/2}.
\]
And, this inequality can also be replaced by $E^+(X)^3 E^\times(Y) \lesssim |A|^{11}$. Both these two inequalities can be compared with \eqref{eq_BW_conj} with $\eta=1/4$. In \cite{Sha18}, the corresponding result is the following.

\begin{theorem} [Shakan] \label{thm_Shakan2}
Let $A\subseteq \bR$ be finite. Then there are sets $X,Y\subseteq A$ with $X\cup Y=A$ and $|X|,|Y|\geq |A|/2$, such that
\begin{equation}\label{eq_d+dtimes}
d_3^+(X)d_3^\times(Y) \lesssim |A|.
\end{equation}
\end{theorem}

\noindent Moreover, the inequality \eqref{eq_d+dtimes} can also be replaced by $d_4^+(X)E^\times(Y) \lesssim |A|^3$. Both $d_3^+(A)$ and $d_3^\times (A)$ take value between $1$ and $|A|$. One can conclude from \eqref{eq_d+dtimes} that, whenever one attains $\gtrsim |A|$, the other is $\lesssim 1$.

In this paper, we prove the following theorem.

\begin{theorem} \label{thm_XY_decomp}
Let $A\subseteq \bR$ be finite. Then there are sets $X,Y\subseteq A$  with $X\cup Y=A$ and $|X|,|Y|\geq |A|/2$, such that
\begin{equation} \label{eq_XY_decomposition}
E_3^+(X)^4 E^\times (Y)^3\lesssim |A|^{22}.
\end{equation}
\end{theorem}

Note that \eqref{eq_XY_decomposition} implies that
\[
E^+(X)^8 E^\times(Y)^3 \lesssim |A|^{30}.
\]
The average exponent on the energies is $3-3/11$, which can be compared with Theorem \ref{thm_BalWoo_improvement}. Indeed, we can improve the average exponent further.

\begin{theorem} \label{thm_XYZ_decomp}
Let $A\subseteq \bR$ be finite. Then there are subsets $X,Y,Z$ of $A$  with $X\cup Y\cup Z=A$, $|X|,|Y|\geq |A|/4$ and $|Z|\geq |A|/2$, such that
\[
E^+(X)^{182}E^\times(Y)^{26}E^\times(Z)^{63}\lesssim |A|^{736}.
\]
\end{theorem}

Now the average exponent on the energies is $736/(182+26+63)=2.71587\ldots$, which is smaller than $3-3/11=2.72727\ldots$. To obtain this exponent, and also obtain improvements on the difference-product/division estimates, we need the flexibility of $d_3^+(A)$. But the term $E_3^+$ in our Theorems \ref{thm_E_3^+} and \ref{thm_XY_decomp} can not be changed to a term in $d_3^+$ directly. A trick we use here would be the application of \eqref{eq_d+dtimes} before inserting the spectral bounds.

\medskip

Another alternative formulation of the sum-product phenomenon is to consider products and products with shifts. That is to say, a strong multiplicative structure of $A$ should be disturbed by additive shifts. A theorem of Garaev and Shen \cite{GarShe10} implies that
\begin{equation} \label{eq_GS}
|AA||(A+1)(A+1)| \gg |A|^{5/4}.
\end{equation}
See Bourgain \cite{Bou05} for the same estimate in the finite fields. In \cite{Shk17}, Shkredov proved that if $A$ is a finite subset of $\bR$ with $|AA|\leq K|A|$ or $|A/A|\leq K|A|$, then
\[
|(A+\alpha)(A+\beta)|\gtrsim K^{-4}|A|^2,\quad |A+\alpha A+\beta A| \gtrsim K^{-6}|A|^2
\]
for any $\alpha,\beta\neq 0$. In particular, when $A$ has strong multiplicative structure, i.e., $|AA|\ll |A|$, then the tight bound $|(A+1)(A+1)|\gtrsim |A|^2$ follows. In this paper, we use the asymmetric estimate in Theorem \ref{thm_asy_1} to prove the following improvements.

\begin{theorem} \label{thm_A+1A+1}
Let $A\subseteq \bR$ be a finite set and $\alpha,\beta$ be non-zero real numbers. Suppose that
\[
|AA| \leq K|A|,\quad \text{ or }\quad |A/A|\leq K|A|.
\]
Then
\begin{equation} \label{eq_A+1A+1}
|(A+\alpha)(A+\beta)| \gtrsim K^{-3}|A|^2,
\end{equation}
and
\[
|A+\alpha A+\beta A| \gtrsim K^{-5}|A|^2.
\]
\end{theorem}

Indeed, the strength of \eqref{eq_A+1A+1}, i.e.,
\[
|AA|^3|(A+1)(A+1)|\gtrsim |A|^5,
\]
is same with that of \eqref{eq_GS}. Theorem \ref{thm_A+1A+1} is a version of ``small product, large product of shifts''. For iterated product of shifts $(A+1)^{(k)}$ with large $k$, we refer readers to \cite{HRZ18,HRZ19}.

For convenience, we always assume that our set $A$ does not contain $0$. In all the theorems and lemmas, the exponent on $\log |A|$, which is abbreviated in `$\lesssim$' or `$\gtrsim$', can be calculated explicitly. However, we do not pursue such accuracy in this paper. We will prove Theorems \ref{thm_asy_1}, \ref{thm_asy_2} in Section 2, Theorems \ref{thm_E_3^+}, \ref{thm_XY_decomp} in Section 3, and all the others in Section 4. A discussion on the regularization lemmas of Rudnev and Stevens \cite{RudSte20} will be held in Section 5. For basics in additive combinatorics, we refer the readers to \cite{TV}.





\section{Two asymmetric estimates}

In this section, we will prove Theorems \ref{thm_asy_1} and \ref{thm_asy_2}. We need the following lemmas deduced from Theorem \ref{thm_ST} (see \cite[Corollary 8.7]{TV}).

\begin{lemma} \label{lem_rich_lines}
If $\cP$ is any finite set of points in $\bR^2$ and $k\geq 2$, then
\[
\#\{l \text{ a line}:\, |l\cap \cP|\geq k\} \ll \frac{|\cP|^2}{k^3}+\frac{|\cP|}{k}.
\]
\end{lemma}

\begin{lemma}\label{lem_rich_points}
If $\cL$ is any finite set of lines in $\bR^2$ and $k\geq 2$, then
\[
\#\{p \text{ a point}:\, \#\{l\in \cL:\, p\in l\}\geq k\} \ll \frac{|\cL|^2}{k^3}+\frac{|\cL|}{k}.
\]
\end{lemma}

Let $A_1,A_2,A_3$ be sets in Theorem \ref{thm_asy_1}, which satisfies $|A_1|\leq |A_2|\leq |A_3|$. Let $\cL$ be the set of lines such that $l\in \cL$ if and only if $l$ contains three distinct points $u_1,u_2,u_3$ with $u_i\in A_i\times A_i$ $(i=1,2,3)$. We have the trivial bound
\[
|\cL| \leq |A_1\times A_1|\cdot |A_2\times A_2| = |A_1|^2|A_2|^2.
\]
For each $l\in \cL$ and $i=1,2,3$, denote $\alpha_{i,l} = |l\cap (A_i\times A_i)|$. It is obvious that
\[
\alpha_{i,l} \leq |A_i| = |A_i\times A_i|^{1/2}.
\]
For $1\leq k\leq |A_i|$, let $\cL_{i,k}$ consist of lines $l\in \cL$ such that $\alpha_{i,l}\geq k$.

\begin{lemma} \label{lem_asy1}
For $i=1,2,3$ and $1\leq p\leq 3$, we have
\[
\sum\limits_{l\in \cL_{i,2}}\alpha_{i,l}^p \lesssim |A_1|^{3-p}|A_i|^{p+1}.
\]
\end{lemma}

\begin{proof}
Note that
\[
\sum\limits_{l\in \cL_{i,2}} \alpha_{i,l}^p =\sum\limits_{2\leq k\leq |A_i|} k^p\cdot \#\{l\in \cL_{i,2}:\, \alpha_{i,l}=k\} \ll \sum\limits_{2\leq k\leq |A_i|} k^{p-1} |\cL_{i,k}|.
\]

For $i=2,3$ and for every point $u\in A_i\times A_i$, there are at most $|A_1|^2$ lines $l\in \cL$ intersecting with $u$. So $\cI(A_i\times A_i,\cL)\leq |A_1|^2|A_i|^2$. On the other hand, one has $\cI(A_i\times A_i,\cL)\geq k|\cL_{i,k}|$. It follows that
\begin{equation} \label{eq_proof_asy1_eq1}
|\cL_{i,k}|\leq \frac{|A_1|^2|A_i|^2}{k}.
\end{equation}
Moreover, Lemma \ref{lem_rich_lines} gives another bound
\begin{equation} \label{eq_proof_asy1_eq2}
|\cL_{i,k}| \ll \frac{|A_i\times A_i|^2}{k^3} = \frac{|A_i|^4}{k^3}
\end{equation}
for $i=1,2,3$ and $2\leq k\leq |A_i|$.

Now for $i=2,3$ and $1\leq p\leq 3$, we deduce by \eqref{eq_proof_asy1_eq1} that
\begin{align*}
\sum\limits_{2\leq k\leq |A_i|/|A_1|} k^{p-1} |\cL_{i,k}| &\ll \sum\limits_{2\leq k\leq |A_i|/|A_1|} k^{p-2} |A_1|^2|A_i|^2 \\
&\lesssim \left(\frac{|A_i|}{|A_1|}\right)^{p-1}|A_1|^2|A_i|^2=|A_1|^{3-p}|A_i|^{p+1}.
\end{align*}
By \eqref{eq_proof_asy1_eq2}, we also have
\begin{align*}
\sum\limits_{|A_i|/|A_1|\leq k\leq |A_i|} k^{p-1} |\cL_{i,k}| &\ll \sum\limits_{|A_i|/|A_1|\leq k\leq |A_i|} k^{p-4}|A_i|^4 \\
&\ll \left(\frac{|A_i|}{|A_1|}\right)^{p-3} |A_i|^4 \lesssim |A_1|^{3-p}|A_i|^{p+1}.
\end{align*}
Then
\[
\sum\nolimits_{l\in \cL_{i,2}} \alpha_{i,l}^p\ll |A_1|^{3-p}|A_i|^{p+1}
\]
for $i=2,3$. Similarly, it follows from \eqref{eq_proof_asy1_eq2} that
\[
\sum\limits_{l\in \cL_{1,2}}\alpha_{1,l}^3 \ll \sum\limits_{2\leq k\leq |A_i|}k^2|\cL_{i,k}| \lesssim \sum\limits_{2\leq k\leq |A_i|}k^{-1}|A_i|^4 \lesssim |A_1|^4.
\]
The proof is completed.


\end{proof}


\begin{proof} [Proof of Theorem \ref{thm_asy_1}]
It is easy to see that
\begin{equation} \label{eq_0}
T^o(A_1,A_2,A_3) \leq \sum\limits_{l\in \cL}\alpha_{1,l} \alpha_{2,l}\alpha_{3,l}.
\end{equation}

Denote $I_1=\{1\}$ and $I_2=\{2,3,4,\ldots\}$. We split the summands in \eqref{eq_0} into several cases: \begin{equation} \label{eq_1}
S_{j_1,j_2,j_3}:=\sum\limits_{l\in \cL,\,\alpha_{1,l}\in I_{j_1}\atop \alpha_{2,l}\in I_{j_2}\, \alpha_{3,l}\in I_{j_3}}\alpha_{1,l} \alpha_{2,l}\alpha_{3,l},
\end{equation}
where $(j_1,j_2,j_3)\in \{1,2\}^3$.

By Lemma \ref{lem_asy1}, we have
\begin{eqnarray*}
S_{1,2,2}&\leq \sum\limits_{l\in \cL_{2,2}\cap \cL_{3,2}} \alpha_{2,l}\alpha_{3,l}\leq \Big(\sum\limits_{l\in \cL_{2,2}}\alpha_{2,l}^{3/2}\Big)^{2/3}\Big(\sum\limits_{l\in \cL_{3,2}}\alpha_{3,l}^{3}\Big)^{1/3}\\
&\lesssim (|A_1|^{3/2}|A_2|^{5/2})^{2/3}(|A_3|^4)^{1/3} = |A_1||A_2|^{5/3}|A_3|^{4/3}.
\end{eqnarray*}
Note that, for $i=2,3$,
\begin{equation} \label{eq_2}
\sum\limits_{l\in \cL \atop \alpha_{i,l}=1} \alpha_{i,l}^3 \leq |\cL| \leq |A_1|^2|A_2|^2 \leq |A_i|^4.
\end{equation}
We deduce also by Lemma \ref{lem_asy1} that, for $j_1=2$, 
\[
S_{2,j_2,j_3}\leq \Big(\prod\limits_{i=1}^3\sum\limits_{l\in \cL\atop \alpha_{i,l}\in I_{j_i}}\alpha_{i,l}^3\Big)^{1/3}\lesssim \Big(\prod\limits_{i=1}^3 |A_i|^4\Big)^{1/3} = |A_1|^{4/3}|A_2|^{4/3}|A_3|^{4/3}.
\]
Moreover,
\begin{align*}
S_{1,1,j_3} &\leq \sum\limits_{l\in \cL,\, \alpha_{3,l}\in I_{j_3}} \alpha_{3,l}\leq |\cL|^{2/3}\Big(\sum\limits_{l\in \cL,\,\alpha_{3,l}\in I_{j_3}}\alpha_{3,l}^{3}\Big)^{1/3}\\
&\lesssim (|A_1|^2|A_2|^2)^{2/3}(|A_3|^4)^{1/3} = |A_1|^{4/3}|A_2|^{4/3}|A_3|^{4/3},
\end{align*}
and
\begin{align*}
S_{1,2,1} &\leq \sum\limits_{l\in \cL_{2,2}} \alpha_{2,l}\leq |\cL|^{2/3}\Big(\sum\limits_{l\in \cL_{2,2}}\alpha_{2,l}^{3}\Big)^{1/3}\\
&\lesssim (|A_1|^2|A_2|^2)^{2/3}(|A_2|^4)^{1/3} = |A_1|^{4/3}|A_2|^{8/3}.
\end{align*}
Now \eqref{eq_asy_1} follows.

Next, let us consider the terms which are counted by $T(A_1,A_2,A_3)$ but not by $T^o(A_1,A_2,A_3)$. For $a_1,a_2\in A_1$, $b_1,b_2\in A_2$, and $c_1,c_2\in A_3$ satisfying
\begin{equation}\label{eq_T_energy}
(b_1-a_1)(c_2-a_2)=(c_1-a_1)(b_2-a_2),
\end{equation}
the triple of points $u_1=(a_1,a_2)$, $u_2=(b_1,b_2)$, $u_3=(c_1,c_2)$ are distinct and collinear provided that $a_1\neq b_1$, $a_1\neq c_1$ and $b_1\neq c_1$. For $a_1=b_1$, there are at most
\[
|A_1\cap A_2\cap A_3||A_1||A_2||A_3|+|A_1\cap A_2|^2|A_3|^2
\]
solutions to \eqref{eq_T_energy}. Similar bounds hold for the cases $a_1=c_1$ or $b_1=c_1$. The proof is completed.
\end{proof}

\begin{remark}
When $|A_3|$ is much bigger than $|A_1|$ and $|A_2|$, one may explore better asymmetric bound than \eqref{eq_asy_1}.
\end{remark}

Next, let us prove Theorem \ref{thm_asy_2}. Note that $R(Z;A_1,A_2)$ is nearly the same with $\sum\nolimits_{z\in Z}r^2(z)$, where we define
\[
r(z)=\#\{(a_1,a_1^\prime,a_2,a_2^\prime)\in A_1\times A_1\times A_2\times A_2:\, (a_1^\prime+a_2^\prime)=z(a_1+a_2)\}.
\]
Indeed, the number of trivial solutions to
\[
(a_1+a_2)(a_1^{\prime\prime\prime}+a_2^{\prime\prime\prime})=(a_1^\prime+a_2^\prime)(a_1^{\prime\prime}+a_2^{\prime\prime}),
\]
with $a_i,a_i^\prime,a_i^{\prime\prime},a_i^{\prime\prime\prime}\in A_i$ $(i=1,2)$, is at most $\textit{O}(|A_1\cap A_2|^2|A_1|^2|A_2|^2)$, which is neglectable when it is compared with \eqref{eq_asy_2}.

We first prove an upper bound for $\sum\nolimits_{z\in Z} r(z)$.

\begin{lemma} \label{lem_asy2}
Let $A_1,A_2$ be two finite subsets of $\bR$ with $|A_1|\leq |A_2|$. Then
\[
\sum\limits_{z\in Z} r(z) \lesssim |Z|^{1/2}|A_1|^{5/3}|A_2|^{4/3}+|Z|^{2/3}|A_1|^{4/3}|A_2|^{4/3}+|Z||A_1|^2.
\]
\end{lemma}

\begin{proof}
Note that
\begin{align}
\sum\limits_{z\in Z}r(z) &= \sum\limits_{z\in Z}\sum\limits_y \#\{(a_1,a_1^\prime,a_2,a_2^\prime)\in A_1\times A_1\times A_2\times A_2:\, a_1^\prime-za_1=za_2-a_2^\prime=y\} \nonumber\\
&=\sum\limits_{(z,y)\in \cP}r_1(z,y)r_2(z,y), \label{eq_proof_asy2_eq1}
\end{align}
where
\[
\cP=\{(z,y):\, z\in Z,\, y\in (A_1-zA_1)\cap (zA_2-A_2)\neq \emptyset\},
\]
and
\[
r_1(z,y)=\{(a_1,a_1^\prime)\in A_1\times A_1:\, a_1^\prime-za_1=y\},
\]
\[
r_2(z,y)=\{(a_2,a_2^\prime)\in A_2\times A_2:\, za_2-a_2^\prime=y\}.
\]
The cardinality of the set $\cP$ can be bounded by
\[
|\cP|\leq |Z||A_1|^2,
\]
and for every $(z,y)\in \cP$, one has $r_1(z,y)\geq 1$ and $r_2(z,y)\geq 1$. We split the right-hand side of \eqref{eq_proof_asy2_eq1} into four sums
\[
S_{j_1,j_2} := \sum\limits_{(z,y)\in \cP\atop {r_1(z,y)\in I_{j_1}\atop r_2(z,y)\in I_{j_2}}} r_1(z,y)r_2(z,y),
\]
where $(j_1,j_2)\in \{1,2\}^2$ and $I_1=\{1\}$, $I_2=\{2,3,4,\ldots\}$.

First, we have
\[
S_{1,1} \leq |\cP| \leq |Z||A_1|^2.
\]
Second,
\[
S_{2,1} \leq \sum\limits_{(z,y)\in \cP}r_1(z,y) \leq \#\{(z,y,a_1,a_1^\prime)\in Z\times \bR\times A\times A:\, a_1^\prime-za_1 =y\} = |Z||A_1|^2.
\]
Third, let us deal with $S_{2,2}$. Applying H\"{o}lder's inequality, we obtain
\begin{align*}
S_{2,2} &= \sum\limits_{(z,y)\in \cP\atop r_1(y,z),r_2(y,z)\geq 2} r_1^{1/2}(y,z)r_1^{1/2}(y,z)r_2(y,z)\\
&\leq \left(\sum\limits_{(z,y)\in \cP}r_1(z,y)\right)^{1/2}\left(\sum\limits_{(z,y)\in \cP\atop r_1(y,z)\geq 2}r_1^3(z,y)\right)^{1/6}\left(\sum\limits_{(z,y)\in \cP\atop r_2(y,z)\geq 2}r_2^3(z,y)\right)^{1/3}.
\end{align*}
Denote by $\cL_1$ the set of lines $l_{a_1,a_1^\prime}:\, a_1^\prime-za_1=y$ with $a_1,a_1^\prime\in A_1$. 
Denote by $\cL_2$ the set of lines $l_{a_2,a_2^\prime}:\, za_2-a_2^\prime=y$ with $a_2,a_2^\prime\in A_2$. 
Note that, for $i=1,2$,
\[
r_i(z,y)\leq |A_i|=|\cL_i|^{1/2}.
\]
For $2\leq k\leq |A_i|$, denote by $\cP_{i,k}$ the set of points in $\cP$ which intersect with at least k lines in $\cL_i$. Applying Lemma \ref{lem_rich_points}, one obtains that
\[
|\cP_{i,k}| \ll \frac{|\cL_i|^2}{k^3} = \frac{|A_i|^4}{k^3}.
\]
Then
\begin{align*}
\sum\limits_{(z,y)\in\cP\atop r_i(z,y)\geq 2}r_i^3(z,y) &= \sum\limits_{2\leq k\leq |A_i|}k^3 \#\{(z,y)\in \cP:\, r_i(z,y)=k\}\\
&\ll  \sum\limits_{2\leq k\leq |A_i|}k^2 |\cP_{i,k}| \ll \sum\limits_{2\leq k\leq |A_i|} k^2\frac{|A_i|^4}{k^3} \lesssim |A_i|^4.
\end{align*}
It follows that
\[
S_{2,2} \lesssim (|Z||A_1|^2)^{1/2}(|A_1|^4)^{1/6}(|A_2|^4)^{1/3} = |Z|^{1/2}|A_1|^{5/3}|A_2|^{4/3}.
\]
Fourth,
\[
S_{1,2} \leq \sum\limits_{(z,y)\in \cP\atop r_2(y,z)\geq 2}r_2(y,z) \leq |\cP|^{2/3}\left(\sum\limits_{(y,z)\in \cP\atop r_2(y,z)\geq 2}r_2^3(y,z)\right)^{1/3} \lesssim |Z|^{2/3}|A_1|^{4/3}|A_2|^{4/3}.
\]
Now the lemma follows.
\end{proof}


\begin{proof} [Proof of Theorem \ref{thm_asy_2}]
For $t\geq 1$, define $Z_t:=\{z\in Z:\, r(z)\geq t\}$. Note that
\[
|Z_t|\leq |Z|\lesssim |A_1|^2.
\]
By Lemma \ref{lem_asy2}, we conclude that
\[
t|Z_t|\leq \sum\limits_{z\in Z_t}r(z) \lesssim |Z_t|^{1/2}|A_1|^{5/3}|A_2|^{4/3},
\]
i.e.,
\[
|Z_t| \lesssim \frac{|A_1|^{10/3}|A_2|^{8/3}}{t^2}.
\]




Note that $r(z)\leq |A_1|^2|A_2|$. We deduce that
\begin{align*}
R(Z;A_1,A_2)&\leq \sum\limits_{z\in Z}r^2(z) = \sum\limits_{t \leq |A_1|^2|A_2|} t^2 \#\{z:\, r(z)=t\} \\
&\ll \sum\limits_{t \leq |A_1|^2|A_2|} t|Z_t| \lesssim \sum\limits_{t \leq |A_1|^2|A_2|}t\frac{|A_1|^{10/3}|A_2|^{8/3}}{t^2} \lesssim |A_1|^{10/3}|A_2|^{8/3}.
\end{align*}
The proof is completed.

\end{proof}

\section{The additive energy of third moment}


In this section, we will prove Theorems \ref{thm_E_3^+} and \ref{thm_XY_decomp}.

\begin{lemma} \label{lemma_E_3^+_basic}
Let $A\subseteq \bR$ be finite. Then there is some $A^\prime\subseteq A$ such that
\[
|A^\prime|\gtrsim E_3^+(A)^{1/2}|A|^{-1}
\]
and
\[
E_3^+(A)^4 E^\times (A^\prime)^3\lesssim |A^\prime|^{12}|A|^{10}.
\]
\end{lemma}

\begin{proof}
At first, let us follow the proof of  \cite[Theorem 2.13]{RSS16} or \cite[Lemma 5.6]{Sha18}. By a standard dyadic pigeonhole argument, there is some $t \leq |A|$ and a set of popular differences
\[
P =\{x\in A-A:\, t\leq r_{A-A}(x)< 2t\}
\]
such that $E_3^+(A) \approx |P|t^3$. By dyadic decomposition again, there is a $q_1$ with $1\leq q_1\leq |A|$ and a set $A_1=\{a\in A:\, q_1 \leq r_{P+A}(a)< 2q_1\}$ of popular abscissae, such that
\[
|A_1|q_1 \approx \sum\limits_{a\in A_1}r_{P+A}(a) \approx \sum\limits_{x\in P}r_{A-A}(x)\approx |P|t.
\]
Again, by another dyadic decomposition, there is a $q_2$ with $1\leq q_2\leq |A_1|$ and a set $A_2=\{b\in A:\, q_2 \leq r_{A_1-P}(b)< 2q_2\}$ of popular ordinates, such that
\[
|A_2|q_2 \approx \sum\limits_{b\in A_2}r_{A_1-P}(b) \approx \sum\limits_{a\in A_1}r_{P+A}(a)\approx |A_1|q_1.
\]
Either $q_1\lesssim |A_1|$ or $q_2\leq |A_1| \lesssim q_1 \lesssim |A_2|$. We will proceed the proof with the case $q_2\lesssim |A_2|$. For the case $q_1\lesssim |A_1|$, the proof is similar (also see the proof of Theorem \ref{thm_E_3^+} below).

Now there are $E^\times(A_2)$ quadruples $(b_1,b_2,b_3,b_4)\in A_2^4$ such that
\[
\frac{b_1}{b_2} = \frac{b_3}{b_4}.
\]
For each given $(b_1,b_2,b_3,b_4)$, there are approximately $q_2^4$ choices of $(a_1,a_2,a_3,a_4)$ such that $a_i\in A_1\cap (P+b_i)$ and
\[
\frac{a_1-(a_1-b_1)}{a_2-(a_2-b_2)} = \frac{a_3-(a_3-b_3)}{a_4-(a_4-b_4)}.
\]
Denote $s_i=a_i-b_i$ $(i=1,2,3,4)$. Then $s_i\in P$. And
\begin{equation} \label{eq_Eq^4leq N}
E^\times(A_2)\left(\frac{|P|t}{|A_2|}\right)^4\approx E^\times (A_2)q_2^4 \ll N,
\end{equation}
where
\begin{eqnarray*}
N=\#\left\{(a_1,a_2,a_3,a_4,s_1,s_2,s_3,s_4)\in A^4\times P^4:\, \frac{a_1-s_1}{a_2-s_2} = \frac{a_3-s_3}{a_4-s_4} \in A_2/A_2\right\}.
\end{eqnarray*}
The difference between our argument and the previous ones is the introduction of the constrain here: the value set $Z:=A_2/A_2$. Denote
\[
r(z)=\#\left\{(a_1,a_2,s_1,s_2)\in A^2\times P^2:\, (a_1-s_1)=z(a_2-s_2)\right\}.
\]
Then
\[
N\leq \sum\limits_{z\in Z}r^2(z).
\]
In \cite{RSS16,Sha18}, the symmetric estimate \eqref{eq_sym_2} is applied, i.e., $N\lesssim |A|^3|P|^3$. This leads to
\[
|P|t^4 \lesssim \frac{|A_2|^4|A|^3}{E^\times (A_2)}.
\]
So, when $|P|\leq |A|$, one obtains
\[
E^+_3(A) \approx |P|t^3 = |P|^{1/4}(|P|t^4)^{3/4} \lesssim |A|^{1/4}\frac{|A_2|^3|A|^{9/4}}{E^\times (A_2)^{3/4}} = \frac{|A_2|^3|A|^{5/2}}{E^\times (A_2)^{3/4}},
\]

For $|P|\geq |A|$, noting that
\[
|Z|= |A_2/A_2|\leq |A_2|^2 \leq |A|^2,
\]
we apply Theorem \ref{thm_asy_2} to obtain
\[
N \lesssim |A|^{10/3}|P|^{8/3}.
\]
Combining \eqref{eq_Eq^4leq N}, one gets
\[
E^\times (A_2) \frac{|P|^4t^4}{|A_2|^4} \lesssim |A|^{10/3}|P|^{8/3},
\]
i.e.,
\[
E_3^+(A)\approx |P|t^3 \lesssim \frac{|A_2|^3|A|^{5/2}}{E^\times(A_2)^{3/4}}.
\]
We conclude that
\[
E_3^+(A)^4 E^\times (A_2)^3  \lesssim |A_2|^{12}|A|^{10}.
\]

Moreover, since $t\leq |A|$, one has
\[
|A_2|^2\gtrsim |A_2|q \approx |P|t \approx \frac{E_3^+(A)}{t^2} \geq \frac{E_3^+(A)}{|A|^2}.
\]
The proof is completed.
\end{proof}

Next we prove Theorem \ref{thm_E_3^+}, which follows directly from the following proposition.

\begin{proposition}
Let $A,B$ be two finite subsets of $\bR$. Suppose that $|A|\lesssim |B|$. Then
\[
|AA| \gtrsim E_3^+(A,B)^{4/3} |B|^{-10/3}.
\]
\end{proposition}

\begin{proof}
The proof is nearly the same with that of Lemma \ref{lemma_E_3^+_basic}. There is some $t \leq |A|$ and a set of popular differences
\[
P =\{x\in A-B:\, t\leq r_{A-B}(x)< 2t\}
\]
such that $E_3^+(A,B) \approx |P|t^3$. By dyadic decomposition again, there is a $q$ with $1\leq q\leq |A|$ and a set $A_1=\{a\in A:\, q \leq r_{P+B}(a)< 2q\}$ of popular abscissae, such that
\[
|A_1|q \approx \sum\limits_{a\in A_1}r_{P+B}(a) \approx \sum\limits_{x\in P}r_{A-B}(x)\approx |P|t.
\]
The symmetric estimate gives
\[
\frac{|A_1|^4}{|AA|}\frac{|P|^4t^4}{|A_1|^4}\lesssim E^\times (A_1)q^4 \lesssim |P|^3|B|^3.
\]
For $|P|\leq |B|$, one obtains that
\[
E_3^+(A,B)\approx |P|^{1/4}(|P|t^4)^{3/4} \lesssim |B|^{1/4} |AA|^{3/4}|B|^{9/4} = |AA|^{3/4}|B|^{5/2}.
\]
When $|P|\geq |B|$, we have $|Z|=|A_1/A_1|\leq |A_1|^2\lesssim |B|^2$. It follows from Theorem \ref{thm_asy_2} that
\[
\frac{|A_1|^4}{|AA|}\frac{|P|^4t^4}{|A_1|^4}\lesssim E^\times (A_1)q^4 \lesssim |P|^{8/3}|B|^{10/3}.
\]
So
\[
E_3^+(A,B) \approx |P|t^3 \lesssim |AA|^{3/4}|B|^{5/2}.
\]
The proof is completed.

\end{proof}



To prove Theorem \ref{thm_XY_decomp}, we need the following Lemma of $l^4$-norm inequality.

\begin{lemma} [Lemma 8 of \cite{KonShk16}] \label{lemma_E_l_4}
Let $A_1,\ldots,A_k\subseteq \bR$ be finite and disjoint. Then
\[
E^\times\Big(\bigcup\limits_{i=1}^k A_i\Big)^{1/4} \leq \sum\limits_{i=1}^k E^\times (A_i)^{1/4}.
\]
\end{lemma}

\begin{proof} [Proof of Theorem \ref{thm_XY_decomp}]
The proof is similar with that of \cite[Theorem 1.10]{Sha18}. Let $A_0=\emptyset$ and we define a sequence of sets $A_1,A_2,\ldots$ by iteration. Suppose that $A_0,A_1,\ldots,A_{j-1}$ has been defined for some $j\geq 1$. Denote $B_{j-1}:=A\setminus (A_0\cup \ldots \cup A_{j-1})$. By Lemma \ref{lemma_E_3^+_basic}, there is a non-empty set $A_j\subseteq B_{j-1}$ such that
\begin{equation} \label{eq_E3BjEA_j}
E_3^+(B_{j-1})^4 E^\times (A_j)^3  \lesssim |A_j|^{12}|B_{j-1}|^{10}.
\end{equation}
Since the sets $A_0,A_1,\ldots$ are nonempty and disjoint, there is some $J\geq 1$ such that
\[
|A_1\cup A_2\cup \ldots \cup A_{J-1}|< \frac{|A|}{2} \leq |A_1\cup A_2\cup \ldots \cup A_J|.
\]
Take
\[
X=B_{J-1},\quad Y=A_1\cup A_2\cup \ldots A_J.
\]
Then $|X|,|Y|\geq |A|/2$ and $X\cup Y=A$. Since $X\subseteq B_{j-1}$ for all $1\leq j\leq J$, it follows from \eqref{eq_E3BjEA_j} that
\[
E_3^+(X)^4 E^\times (A_j)^3  \lesssim |A_j|^{12}|B_{j-1}|^{10}.
\]
By Lemma \ref{lemma_E_l_4}, one deduces that
\begin{align*}
E_3^+(X)^4 E^\times (Y)^3 &= E_3^+(X)^4 E^\times \Big(\bigcup\limits_{j=1}^J A_j\Big)^3 \leq E_3^+(X)^4 \left(\sum\limits_{j=1}^J E^\times (A_j)^{1/4}\right)^{12}\\
&= \left(\sum\limits_{j=1}^J E_3^+(X)^{1/3} E^\times (A_j)^{1/4}\right)^{12}\lesssim \left(\sum\limits_{j=1}^J |A_j||B_{j-1}|^{5/6}\right)^{12} \\
&\leq |A|^{10}\left(\sum\limits_{j=1}^J|A_j|\right)^{12} \leq |A|^{22}.
\end{align*}
In the last two steps, we have used the inequalities $|B_{j-1}|\leq |A|$ and $\sum\limits_{j=1}^J|A_j|=|Y|\leq |A|$.
\end{proof}

\section{The sum-product type problems}

In this section, Theorems \ref{thm_difference}, \ref{thm_BalWoo_improvement}, \ref{thm_XYZ_decomp}  and \ref{thm_A+1A+1} will be proved. We quote the following results.

\begin{lemma} [Theorem 11 of \cite{Shk15}] \label{lem_d_3_AA}
Let $A\subseteq \bR$ be finite. Then
\[
|A+A| \gtrsim |A|^{58/37}d^+_3(A)^{-21/37}, \quad |AA| \gtrsim |A|^{58/37}d^\times_3(A)^{-21/37}.
\]
\end{lemma}

\begin{lemma} [Corollary 4.6 of \cite{KonRud13}] \label{lem_d_3_A/A}
Let $A\subseteq \bR$ be finite. Then
\begin{equation} \label{eq_lem_E3_EA-A}
E_3^+(A)E^+(A,A-A)|A-A|\gg |A|^8.
\end{equation}
In particular,
\[
|A-A| \gtrsim |A|^{8/5}d^+_3(A)^{-3/5},\quad |A/A| \gtrsim |A|^{8/5}d^\times_3(A)^{-3/5}.
\]
\end{lemma}

\begin{lemma} [Theorem 5.4 of \cite{Shk13}]\label{lem_d3E+}
Let $A\subseteq \bR$ be finite. There are real numbers $\Delta,\tau>0$ and sets $D$, $S_\tau$ such that
\[
E^+(A)\approx |D|\Delta^2,\quad \Delta\lesssim \frac{E_3^+(A)}{E^+(A)}, \quad|S_\tau|\lesssim \frac{E_3^+(A)}{\tau^3},
\]
and
\[
\frac{E^+(A)^6}{|A|^6} \lesssim E_3^+(A)\Delta^3 \cdot \sup\limits_\tau\min\left\{\Delta \tau E^+(A,D), \,\tau^2 E^+(A,D)^{1/2}E^+(A,S_\tau)^{1/2}\right\}.
\]
In particular,
\[
E^+(A) \lesssim d^+_3(A)^{7/13}|A|^{32/13}, \quad E^\times(A) \lesssim d^\times_3(A)^{7/13}|A|^{32/13}.
\]
\end{lemma}

\begin{proof} [Proof of Theorem \ref{thm_difference}]
By Theorem \ref{thm_XY_decomp}, there are subsets $X,Y$ of $A$ with $|X|,|Y|\geq |A|/2$ such that
\[
E_3^+(X)^4 E^\times (Y)^3 \lesssim |A|^{22}.
\]
We apply Theorem \ref{thm_Shakan2} to the set $X$. Then there are subsets $X^\prime$, $Y^\prime$ of $X$ with $|X^\prime|,|Y^\prime|\geq |X|/2$ such that
\begin{equation} \label{eq_pf_difference_eq0}
d_3^+(X^\prime)^2d_3^\times (Y^\prime)^2 \lesssim |X|^2.
\end{equation}
Since $X^\prime\subseteq X$, we have
\begin{equation} \label{eq_pf_difference_eq1}
E_3^+(X^\prime)^4 E^\times (Y)^3 \lesssim |A|^{22}.
\end{equation}

Note that
\begin{equation} \label{eq_pf_difference_eq2}
E^+(X^\prime,X^\prime-X^\prime)^4 \leq |X^\prime|^2|X^\prime-X^\prime|^2 E_3^+(X^\prime,X^\prime-X^\prime)^2 \leq |X^\prime|^4|X^\prime-X^\prime|^6 d_3^+(X^\prime)^2.
\end{equation}
By Lemmas \ref{lem_d_3_AA} and \ref{lem_d_3_A/A}, we have
\begin{equation} \label{eq_pf_difference_eq3}
d_3^\times(Y^\prime)^{-2} \lesssim |Y^\prime Y^\prime|^{74/21}|Y^\prime|^{-116/21},\quad d_3^\times(Y^\prime)^{-2} \lesssim |Y^\prime/Y^\prime|^{10/3}|Y^\prime|^{-16/3}.
\end{equation}
Recall \eqref{eq_lem_E3_EA-A}, i.e.,
\begin{equation} \label{eq_pf_difference_eq4}
E_3^+(X^\prime)E^+(X^\prime,X^\prime-X^\prime)|X^\prime-X^\prime| \gg |X^\prime|^8.
\end{equation}
One also has
\begin{equation} \label{eq_pf_difference_eq5}
E^\times (Y) \geq \frac{|Y|^4}{|YY|},\quad E^\times (Y) \geq \frac{|Y|^4}{|Y/Y|}.
\end{equation}
Now we put \eqref{eq_pf_difference_eq0}-\eqref{eq_pf_difference_eq5} all together. Calculation leads to
\[
|A-A|^{10}|AA|^{137/21}\gtrsim |A|^{452/21},\quad |A-A|^{10}|A/A|^{19/3}\gtrsim |A|^{64/3}.
\]
It follows that
\[
\max\{|A-A|,\, |AA|\}\gtrsim |A|^{452/347},\quad \max\{|A-A|,\, |A/A|\}\gtrsim |A|^{64/49}.
\]
Now Theorem \ref{thm_difference} follows.
\end{proof}

\begin{proof} [Proof of Theorem \ref{thm_BalWoo_improvement}]
The proof is similar with that of \cite[Corollary 21]{KonShk16}. Set $B_1=A$ and $C_0=\emptyset$. We use iteration to construct two sequences of sets
\[
B_1\supseteq B_2\supseteq B_3 \supseteq\ldots,\quad C_0\subseteq C_1\subseteq C_2\subseteq \ldots.
\]

Suppose that $B_j$ and $C_{j-1}$ has been constructed for some $j\geq 1$. If $E_3^+ (B_j) > |A|^4/M$, we apply Lemma \ref{lemma_E_3^+_basic} to obtain a subset $D_j$ of $B_j$ such that
\[
|D_j| \gtrsim \frac{E_3^+(B_j)^{1/2}}{|B_j|} \geq \frac{|A|}{M^{1/2}},
\]
and
\[
E^\times (D_j) \lesssim \frac{|D_j|^4|B_j|^{10/3}}{E_3^+(B_j)^{4/3}} \lesssim \frac{|D_j|^4M^{4/3}}{|A|^2}.
\]
Now we set $B_{j+1}=B_j\setminus D_j$, $C_j=C_{j-1}\sqcup D_j$ and proceed the iteration with $j+1$ instead of $j$. If $E_3^+ (B_j) \leq  |A|^4/M$ for some $j=K$, then we stop the iteration and take $B=B_{K}$, $C=C_{K-1}$. By the construction, we have $A=B_j\sqcup C_{j-1}$ for $1\leq j\leq K$. 

Combining Lemma \ref{lemma_E_l_4}, we conclude that

\begin{eqnarray*}
E^\times(C) &= E^\times\big(\bigcup\limits_{j=1}^{K-1} D_j\big)\lesssim \left(\sum\limits_{j=1}^{K-1} E^\times(D_j)^{1/4}\right)^4 \\
&\lesssim \frac{M^{4/3}}{|A|^2}\left(\sum\limits_{j=1}^{K-1} |D_j|\right)^4 \leq  M^{4/3}|A|^2,
\end{eqnarray*}
since $D_1,\ldots,D_{K-1}$ are disjoint and $\sum\nolimits_{j=1}^{K-1} |D_j|\leq |A|$.

Next, one deduces from $E^+(B)\leq |A|E_3^+(B)^{1/2}$ that
\[
E^+(B) \lesssim |A|\left(\frac{|A|^4}{M}\right)^{1/2}\leq \frac{|A|^3}{M^{1/2}}.
\]
Now we optimise over $M$ by taking $M=|A|^{6/11}$, which leads to
\[
E^+(B),\,E^\times (C)\,\lesssim\, |A|^{3-3/11}.
\]
\end{proof}

As for Theorem \ref{thm_XYZ_decomp}, we need to prove another lemma.

\begin{lemma} \label{lem_E3E2+E2times}
Let $A\subseteq \bR$ be finite. Then there are subsets $X,Y$ of $A$ with $X\cup Y=A$ and $|X|,|Y|\geq |A|/2$, such that
\[
E_3^+(X) \gtrsim E^+(X)^{13/6}E^\times(Y)^{13/42}|A|^{-137/42}.
\]
\end{lemma}
\begin{proof}
By Theorem \ref{thm_Shakan2}, there are subsets $X,Y$ of $A$ such that $X\cup Y=A$, $|X|,|Y|\geq |A|/2$ and $d_3^+(X)d_3^\times(Y)\lesssim |A|$. We apply Lemma \ref{lem_d3E+} to the set $X$. Then there are real numbers $\Delta,\tau$ and sets $D$, $S_\tau$ such that 
\[
\frac{E^+(X)^6}{|X|^6} \lesssim E_3^+(X) \Delta^3 \cdot\sup\limits_\tau\min\{\Delta \tau E^+(X,D), \tau^2 E^+(X,D)^{1/2}E^+(X,S_\tau)^{1/2}\}.
\]
Note that
\[
E^+(X,B) \ll |X||B|^{3/2}d^+_3(X)^{1/2}
\]
for any set $B$. Combining the bound $|S_\tau|\lesssim \tau^{-3}E_3^+(X)$, one has
\[
\frac{E^+(X)^6}{|X|^6} \lesssim E_3^+(X) \Delta^3 \cdot\sup\limits_\tau\min \left\{\Delta \tau |X||D|^{3/2}d_3^+(X)^{1/2},\, \tau^{-1/4} |X||D|^{3/4}E_3^+(X)^{3/4}d_3^+(X)^{1/2}\right\}.
\]
The supremum over $\tau$ can be bounded by the value at $\tau=\Delta^{-4/5}|D|^{-3/5}E_3^+(X)^{3/5}$. So
\[
E^+(X)^6\lesssim |X|^6\cdot E_3^+(X) \Delta^3 \cdot\Delta^{1/5} |X| |D|^{9/10} E_3^+(X)^{3/5}d_3^+(X)^{1/2}.
\]
Next, insert the bound $|D|\Delta^2 \approx E^+(X)$ and $\Delta \lesssim E_3^+(X)/E^+(X)$, one deduces that
\[
E^+(X)^{13/2}\lesssim |X|^7 E_3^+(X)^3 d_3^+(X)^{1/2}.
\]
Now, the last inequality in Lemma \ref{lem_d3E+} shows that
\[
d_3^+(X) \lesssim \frac{|A|}{d_3^\times (Y)} \lesssim \frac{|A|^{39/7}}{E^\times(Y)^{13/7}}.
\]
Hence
\[
E_3^+(X) \gtrsim E^+(X)^{13/6}E^\times(Y)^{13/42}|A|^{-137/42}.
\]
The lemma then follows.
\end{proof}

\begin{proof} [Proof of Theorem \ref{thm_XYZ_decomp}]
Applying Theorem \ref{thm_XY_decomp}, one obtains subsets $W,Z$ of $A$ such that $A=W\cup Z$, $|W|,|Z|\geq |A|/2$ and
\[
E_3^+(W)^4 E^\times(Z)^3\lesssim |A|^{22}.
\]
By Lemma \ref{lem_E3E2+E2times}, we further obtain subsets $X,Y$ of $W$ with $W=X\cup Y$ and $|X|,|Y|\geq |W|/2$, such that
\[
E_3^+(X) \gtrsim E^+(X)^{13/6}E^\times(Y)^{13/42}|W|^{-137/42}.
\]
Then $A=X\cup Y\cup Z$ and $|X|,|Y|\geq |A|/4$, $|Z|\geq |A|/2$. Moreover, noting that $E_3^+(X)\leq E_3^+(W)$, we conclude that
\[
E^+(X)^{182}E^\times(Y)^{26}E^\times(Z)^{63}\lesssim |A|^{736}.
\]
The proof is completed.
\end{proof}

To prove Theorem \ref{thm_A+1A+1}, we quote the following lemma.

\begin{lemma} [Lemma 4 of \cite{Shk17}] \label{lem_exists_zA}
Let $A\subseteq \bR$ be a finite set. Then there is a $z$ such that
\[
\sum\limits_{x\in zA}|zA\cap x(zA)|\gg \frac{E^\times(A)}{|A|}.
\]
\end{lemma}

\begin{proof} [Proof of Theorem \ref{thm_A+1A+1}]
The arguments are the same as in the proof of \cite[Theorem 12]{Shk17}. We only sketch the main steps here.

Without loss of generality, we assume that $|AA|\leq K|A|$. Let $\alpha,\beta\neq 0$. By Theorem \ref{thm_asy_1}, we have
\begin{eqnarray*}
&\sum\limits_{a,b \in A} E^\times((AA/a)+\alpha,(AA/b)+\beta) = \sum\limits_{a,b \in A} E^\times((AA/\alpha)+a,(AA/\beta)+b)\\
&\quad\quad \quad  = T(A,AA/\alpha,AA/\beta) \lesssim |A||AA|^3 \leq K^3|A|^4.
\end{eqnarray*}

Note that $A\subseteq AA/a$ for any $a\in A$. There are some $a,b\in A$ such that
\[
E^\times(A+\alpha,A+\beta)\leq E^\times \left(AA/a+\alpha, AA/a^\prime +\beta\right) \lesssim K^3 |A|^2.
\]
In particular, one has
\[
|(A+\alpha)(A+\beta)| \gg \frac{|A+\alpha|^2|A+\beta|^2}{E^\times (A+\alpha,A+\beta)} \gtrsim K^{-3}|A|^2.
\]

Moreover, Lemma \ref{lem_exists_zA} shows that there is some $z\in \bR$ such that
\[
\sum\limits_{c\in zA}|zA\cap c(zA)| \gg \frac{E^\times (A)}{|A|}\gg \frac{|A|^2}{K}.
\]
Let
\[
n(x) = \#\{(c,d)\in (zA)\times (zA):\, cd \in zA,\, (c+\alpha)(d+\beta)=x\}.
\]
Then
\begin{align*}
\frac{|A|^4}{K^2} &\ll \left(\sum\limits_{c\in zA}|zA\cap c(zA)|\right)^2\\
 &= \left(\sum\limits_{x\in zA+\alpha zA+\beta zA+\alpha \beta} n(x)\right)^2 \leq |A+\alpha A+\beta A| \sum\limits_{x}n^2(x).
\end{align*}
On the other hand,
\begin{align*}
\sum\limits_{x}n^2(x) &= \#\{(c_1,d_1,c_2,d_2)\in (zA)^4:\, c_1d_1\in zA,\, c_2d_2\in zA,\, (c_1+\alpha)(d_1+\alpha)=(c_2+\alpha)(d_2+\alpha)\}\\
&\leq E^\times (A+\alpha/z,A+\beta/z)\lesssim K^3|A|^2.
\end{align*}
The theorem now follows by combining the above two inequalities.
\end{proof}




\section{Regularization lemmas and remarks}

At the same time when this paper was written out, Rudnev and Stevens \cite{RudSte20} showed some regularization lemma which has particular interest in the study of Balog-Wooley decomposition. We quote the following lemma with its proof, which is a further exploration of \cite[Lemma 1]{RudSte20} and is provided by Rudnev via personal communication. 

\begin{lemma} [Rudnev-Stevens] \label{lem_regularization}
Let $k> 1$ be a given integer. Let $A$ be a finite subset of $\bR$. Then there are sets $B,B^{\prime\prime}$ with $B^{\prime\prime}\subseteq B\subseteq A$ and $|B^{\prime\prime}|\gtrsim_k |B|\gg_k |A|$, such that the following property holds: there is a number $1\leq t\leq |B|$ and a set $P=\{x\in B-B:\, t\leq r_{B-B}(x)<2t\}$ such that $E_k^+(B) \approx_k |P|t^k$ and $r_{P+B}(b)\approx_k |P|t|A|^{-1}$ for any $b\in B^{\prime\prime}$.
\end{lemma}

Here, the subscripts in $\gtrsim_k$, $\gg_k$ or $\approx_k$ means that the implied constant may depend on $k$.

\begin{proof}
Denote $A_0=A$ and we find sets $A_1,A_2,\ldots$ by iteration. For $i\geq 0$, we can obtain by standard arguments a number $1\leq t_i\leq |A_i|$ and a set $P_i=\{x\in A_i-A_i:\, t_i\leq r_{A_i-A_i}(x)< 2 t_i\}$ such that $ |P_i|t_i^k \leq E_k^+(A_i) < |P_i|(2t_i)^k \log |A_i|/\log 2$. Then the set $G_i:=\{(a,b)\in A_i\times A_i:\, a-b\in P_i\}$ satisfies that $|P_i|t_i\leq |G_i|< 2|P_i|t_i$. Take $\varepsilon = (\log 2)(k2^{k+1}\log^2 |A|)^{-1}$, define a set of abscissae
\[
A_i^\prime = \left\{a\in A_i:\, r_{P_i+A_i}(a)\leq \frac{|G_i|}{\varepsilon |A_i|}\right\}.
\]
Then
\[
|A_i\setminus A_i^\prime| \cdot\frac{|G_i|}{\varepsilon |A_i|} < \#\{(a,b)\in G_i:\, a\notin A_i^\prime\}\leq |G_i|.
\]
It follows that $|A_i\setminus A_i^\prime| < \varepsilon |A_i|$, and then $|A_i^\prime|> (1-\varepsilon)|A_i|$.

Denote $G_i^\prime:=\{(a,b)\in G_i:\, a\in A_i^\prime\}$. If $|G_i^\prime|< |G_i|/2^k$ for some $i\geq 0$, then we set $A_{i+1}=A^\prime_i$ and proceed with $i+1$ instead of $i$, otherwise we set $B=A_i$, $B^\prime=A^\prime_i$ and terminate the process. That is to say, we discard a small set with proportion $\varepsilon$ each time.

Put $I_0=[\varepsilon^{-1}]$. Now we show that this process will stop within $I \leq I_0$ steps. Suppose that this is not the case, then $|G_i^\prime|< |G_i|/2^k$ for all $0\leq i\leq I_0$. Note that $|A_{I_0}|\gg (1-\varepsilon)^{\varepsilon^{-1}}|A| \gg_k |A|$. On the other hand, let us compare the energy-terms in $E_k^+(A_i)$ supported on $(G_i^\prime)^k$ and on $G_i^k$ for each $0\leq i\leq I_0-1$:
\begin{eqnarray*}
&\sum\limits_{x\in P_i}\#\{(a,b)\in A_i^\prime \times A_i:\, a-b=x\}^k \leq (2t_i)^{k-1} |G_i^\prime|\\
&\quad\quad\quad\quad\quad <\frac{t_i^{k-1}|G_i|}{2}  \leq  \frac{1}{2}\sum\limits_{x\in P_i}\#\{(a,b)\in A_i \times A_i:\, a-b=x\}^k.
\end{eqnarray*}
Now we count the number of terms in $E_k^+(A_i)$ which we would like to discard in each step of the iterating process, i.e.,
\begin{eqnarray*}
&&\#\left\{(a_1,b_1,\ldots,a_k,b_k)\in G_i^k\setminus (G_i^\prime)^k:\, a_1-b_1=\ldots =a_k-b_k\right\}\\
&&= \sum\limits_{x\in P_i}\#\{(a,b)\in A_i \times A_i:\, a-b=x\}^k-\sum\limits_{x\in P_i}\#\{(a,b)\in A_i^\prime \times A_i:\, a-b=x\}^k\\
&&\geq \frac{1}{2} \sum\limits_{x\in P_i}\#\{(a,b)\in A_i \times A_i:\, a-b=x\}^k \geq \frac{1}{2}|P_i|t_i^k >\frac{(\log 2) E_k^+(A_i)}{2^{k+1}\log |A_i|}.
\end{eqnarray*}
We emphasis that, any discarded energy-term $(a_1,b_1,\ldots,a_k,b_k)$ has at least one component $(a_j,b_j)$ with abscissa not in $A_i^\prime$. So the energy-terms counted by $E_k^+(A_i^\prime)$ all remains. We deduce that
\begin{align*}
E_k^+(A_{i+1})=E_k^+(A_i^\prime) &= \#\left\{(a_1,b_1,\ldots,a_k,b_k)\in (A_i^\prime\times A_i^\prime)^k:\, a_1-b_1=\ldots =a_k-b_k\right\}\\
&\leq \#\left\{(a_1,b_1,\ldots,a_k,b_k)\in G_i^k:\, a_1-b_1=\ldots =a_k-b_k\right\}\\
&\quad -\#\left\{(a_1,b_1,\ldots,a_k,b_k)\in G_i^k\setminus (G_i^\prime)^k:\, a_1-b_1=\ldots =a_k-b_k\right\}\\
&\leq \left(1-\frac{\log 2}{2^{k+1}\log |A|}\right)E_k^+(A_i)
\end{align*}
for all $0\leq i\leq I_0-1$. Recalling that $\varepsilon = (\log 2)(k2^{k+1}\log^2 |A|)^{-1}$, we obtain
\[
|A|^2\ll_k |A_{I_0}|^2 \leq E_k^+(A_{I_0}) \ll \left(1-\frac{\log 2}{2^{k+1}\log |A|}\right)^{\varepsilon^{-1}} \cdot E_k^+(A) \ll |A|^{-k}|A|^{k+1}\ll |A|,
\]
which is a contradiction.

When the process stops, say, at the $I$-th step, with $B=A_I$ and $B^\prime=A_I^\prime$. Then $|B|\gg_k |A|$. Set
\[
B^{\prime\prime} = \{x\in B^\prime:\, r_{P_I+B}(a)\geq |G_I|/(2^{k+1}|B|)\}.
\]
Then
\[
\#\{(a,b)\in G_I:\, a\in B^\prime\setminus B^{\prime\prime}\} \leq \frac{|G_I|}{2^{k+1}|B|} |B^\prime| \leq \frac{|G_I|}{2^{k+1}}.
\]
It follows that
\[
\#\{(a,b)\in G_I:\, a\in B^{\prime\prime}\} \geq |G_I^\prime| - \frac{|G_I|}{2^{k+1}} \geq \frac{|G_I|}{2^{k}}-\frac{|G_I|}{2^{k+1}} \gg_k |G_I|.
\]
On the other hand, since $B^{\prime\prime}\subseteq B^\prime$, recalling the definition of $B^\prime$, we have
\[
\#\{(a,b)\in G_I:\, a\in B^{\prime\prime}\} \leq |B^{\prime\prime}|\frac{|G_I|}{\varepsilon |B|}.
\]
Hence $|B^{\prime\prime}|\gg_k \varepsilon |B| \gtrsim_k |A|$. The proof is completed.
\end{proof}

Now let us see the effect of above regularization lemma on the relaxed version of Balog-Wooley decomposition.

\begin{proposition}
Let $A\subseteq \bR$ be finite. Then there are sets $B^{\prime\prime}\subseteq B\subseteq A$  with $|B^{\prime\prime}|\gtrsim |B|\gg |A|$, such that
\[
E_3^+(B)^4 E^\times (B^{\prime\prime})^3\lesssim |A|^{22}.
\]
\end{proposition}

\begin{proof}
Given the set $A$ and the moment $k=3$ of the additive energy, we obtain sets $B$, $B^{\prime\prime}$, $P$ and the number $t$ as in Lemma \ref{lem_regularization}. Then we proceed as in the proof of Lemma \ref{lemma_E_3^+_basic}. Recall that $E_3^+(B)\approx |P|t^3$ and $r_{P+B}(b) \approx |P|t|A|^{-1}$ for any $b\in B^{\prime\prime}$. It follows that
\[
E^\times (B^{\prime\prime}) \left(\frac{|P|t}{|A|}\right)^4 \lesssim
\begin{cases}
|P|^3|B|^3,\quad &\text{ in general},\\
|P|^{8/3}|B|^{10/3},\quad &\text{ if } |P|\geq |B|.
\end{cases}
\]
The conclusion follows by similar arguments.
\end{proof}

\subsection*{Acknowledgements}
The author is especially grateful to Professor Misha Rudnev, who kindly shared his experiences on asymmetric estimates and provided the details of the regularization lemmas. The author also thanks Zhenchao Ge for acknowledging important information. This work is supported by National Natural Science Foundation of China (Grant No. 11701549).


\end{document}